\documentclass{amsart}

\usepackage{amsfonts,amsmath,amsxtra,amsthm,tikz}
\usepackage{tikz-cd}
\usetikzlibrary{arrows} 
\usepackage{mathrsfs}
\usepackage{times}
\usepackage{anysize}
\usepackage{graphicx}
\usepackage{youngtab}
\usepackage[utf8]{inputenc}
\usepackage{color}
\usepackage{hyperref}
\usepackage{subfigure}
\usepackage{stmaryrd, MnSymbol}


\marginsize{1in}{1in}{1in}{1in}

\newtheorem{lemma}{Lemma}[section]
\newtheorem{theorem}[lemma]{Theorem}%[section]
\newtheorem{conjecture}[lemma]{Conjecture}%[section]
\newtheorem{corollary}[lemma]{Corollary}%[section]
\newtheorem{proposition}[lemma]{Proposition}%[section]
\theoremstyle{definition}   %\MV{seeing how it looks with this style}
\newtheorem{example}[lemma]{Example}%[section]
\newtheorem{remark}[lemma]{Remark}%[section]
\newtheorem{definition}[lemma]{Definition}%[section] 
\DeclareMathOperator{\spann}{span}
\DeclareMathOperator{\Hilb}{Hilb}
\DeclareMathOperator{\pt}{\mathrm{pt}}
\DeclareMathOperator{\Weyl}{\mathrm{Weyl}}
\DeclareMathOperator{\id}{\mathrm{id}}
\DeclareMathOperator{\Var}{\mathrm{Var}}
\newcommand{\tila}{\tilde{a}}
\newcommand{\tilb}{\tilde{b}}
\newcommand{\tilA}{\tilde{A}}
\newcommand{\tilB}{\tilde{B}}
\newcommand{\tilk}{\tilde{k}}
\DeclareMathOperator{\im}{\mathrm{Im}}

\newcommand{\Ext}{\mathrm{Ext}}
\newcommand{\cO}{\mathcal{O}}
\newcommand{\cI}{\mathcal{I}}
\newcommand{\cN}{\mathcal{cN}}

\newcommand{\Cn}{C^{[n]}}
\newcommand{\cCn}{\mathcal{C}^{[n]}}
\newcommand{\cCfn}{\mathcal{C}^{[n,n+1]}}
\newcommand{\pr}{\prime}
\newcommand{\cX}{\mathcal{X}}
\newcommand{\cC}{\mathcal{C}}
\renewcommand{\hom}{\mathrm{Hom}}
\newcommand{\Sym}{\mathrm{Sym}}

\newcommand{\C}{{\mathbb C}}
\newcommand{\R}{{\mathbb R}}

\newcommand{\A}{\mathbb{A}}
\newcommand{\Z}{{\mathbb Z}}

\newcommand{\Q}{{\mathbb Q}}

\renewcommand{\P}{\mathbb{P}}

\newcommand{\omitt}[1]{}  %for comments or changes to be invisible

\definecolor{Red}{rgb}{0.6,0,0.1}

\author{Oscar Kivinen} 
\address{Department of Mathematics, University of California, Davis.
One Shields Avenue, 95616 Davis, CA}
\email{okivinen@ucdavis.edu}

\title{Homology of Hilbert schemes of reducible locally planar curves}
\date{\today}

\begin{document}
\begin{abstract}
Let $C$ be a complex, reduced, locally planar curve. We extend the results of Rennemo \cite{R} to reducible curves by 
constructing an algebra $A$ acting on $V=\bigoplus_{n\geq 0} H_*(\Cn, \Q)$, where $\Cn$ is the Hilbert scheme of $n$ points on $C$. 
If $m$ is the number of irreducible components of $C$, we realize $A$ as a subalgebra of the Weyl algebra of $\A^{2m}$.
We also compute the representation $V$ in the simplest reducible example of a node. 
\end{abstract}

\maketitle
\tableofcontents

%TODO
%Lopuks esimerkki.

\section{Introduction}\label{sec:introduction}
Let $C$ be a complex, reduced, locally planar curve. We are interested in 
studying the homologies of the Hilbert schemes of points $C^{[n]}$. In the 
case when $C$ is integral, work of Rennemo, Migliorini-Shende, Maulik-Yun 
\cite{R,MS,MY} relates these homologies to the homology of the compactified 
Jacobian of $C$ equipped with the perverse filtration. Furthermore, work of 
Migliorini-Shende-Viviani \cite{MSV} considers an extension of these 
results to reduced but possibly reducible curves. 

Following Rennemo, we approach the problem of computing the homologies of 
the Hilbert schemes in question from the point of view of representation 
theory. In \cite{R}, a Weyl algebra in two variables acting on $V:=
\bigoplus_{n\geq 0} H_*(C^{[n]})$ was constructed for integral locally 
planar curves, and $V$ was described in terms of the representation theory 
of the Weyl algebra. When $C$ has $m$ irreducible components, we construct 
an algebra $A$ acting on $V$, where $A$ is an explicit subalgebra of the 
Weyl algebra in $2m$ variables. The main result is the following.
\begin{theorem}\label{thm:mainthm}
If $C=\bigcup_{i=1}^mC_i$ is a decomposition of $C$ into irreducible
components, the space $V=\bigoplus_{n\geq 0}H_*(C^{[n]},\Q)$ carries a bigraded action of the algebra $A=A_m:=\Q[x_1,\ldots, x_m, \partial_{y_1},\ldots, \partial_{y_m}, \sum_{i=1}^m y_i, \sum_{i=1}^m \partial_{x_i}]$, where $V=\bigoplus_{n,d\geq 0} V_{n,d}$ is graded by number of points $n$ and homological degree $d$. Moreover, the operators $x_i$ have degree $(1,0)$ and the operators $\partial_{y_i}$ have degree $(-1,-2)$ in this bigrading. In effect, the operator $\sum y_i$ has degree $(1,2)$ and the operator $\sum \partial_{x_i}$ has degree $(-1,0)$.
\end{theorem}
\begin{remark}
The algebra $A$ does not depend on $C$, but only on the number of components $m$.
\end{remark}
\begin{remark}
    An argument similar to \cite[Theorem 1.2]{R} shows that $V$ is free over
    $\Q[x_i]$ for any $i=1,\ldots,m$, and also over $\Q[\sum_{i=1}^my_i]$. Through the ORS conjectures (see below), this may be seen as a version of Rasmussen's remark in 
    \cite{RasDiff} that the triply-graded homology of the link $L$ of $C$ is free over the homology of an unlink corresponding to a component of $L$.
\end{remark}

In Section \ref{sec:geometry} we will discuss the relevant geometry, namely the deformation theory of locally planar curves. In particular, we prove that the relative families of (flag) Hilbert schemes have smooth total spaces, which is crucial for applying a bivariant homology formalism, described in Section \ref{sec:bivariant}.

We then define the action of the generators of $A$ on $V$ by explicit geometric constructions in Section \ref{sec:definitions}, and prove the commutation relations in Section \ref{sec:commrel}. 

In Section \ref{sec:node}, we describe the representation $V$ of the algebra $A_2$ in the example of the node.
More precisely, we have 
\begin{theorem}\label{thm:node}
When $C=\{x^2=y^2\}\subset \P^2$, we have that 
\begin{equation}
V\cong \frac{\Q[x_1,x_2,y_1,y_2]}{\C[x_1,x_2,y_1+y_2](x_1-x_2)}
\end{equation}
as an $A$-module, where $$A=\Q[x_1,x_2,\partial_{x_1}+\partial_{x_2},y_1+y_2,\partial_{y_1},\partial_{y_2}]\subset \Weyl(\A^{4}).$$
\end{theorem}

\begin{remark}
 Although seeing the algebra $A$ for the first time immediately raises the question whether we can define the operators $\partial_{x_i}$ or multiplication by $y_i$ separately, i.e. extend this action to the whole Weyl algebra, this example shows that it is in fact not possible to do this while retaining the module structure for $V$.     
\end{remark}

Locally planar curve singularities are connected naturally to topics ranging from the Hitchin fibration \cite{Ngo} to HOMFLY-PT homology of the links of the singularities \cite{ORS,GORS}.  For example, 
from \cite{ORS} we have
\begin{conjecture}
If $C$ has a unique singularity at $0$, its {\em link} is by definition the intersection of $C$ with a small three-sphere around $0$. There is an isomorphism 
$$V_0^c \cong HHH_{a=0}(\text{Link of } C),$$ where $V_0^c=\bigoplus_{n\geq 0} H^*(C_0^{[n]})$ is the cohomology of the punctual Hilbert scheme, and $HHH(-)$ is the triply graded HOMFLY-PT homology of Khovanov and Rozansky
\cite{KR}.
\end{conjecture}

This conjecture is still wide open. Recently, advances on the knot homology side have been made by Hogancamp, Elias and Mellit \cite{EH,Hog,Mel}, who compute the HOMFLY-PT homologies of for example $(n,n)$-torus links using algebraic techniques.
As the $(n,n)$-torus links appear as the links of the curves $C=\{x^n=y^n\}\subset \P^2$, a
partial motivation for this work was to study the Hilbert schemes of points on these curves. 

\begin{remark}
There are many natural algebras acting on $HHH(-)$, for example the positive half of the Witt algebra as proven in \cite{KR2}. It might be possible that the actions of the operators $\mu_+=\sum_i \partial_{x_i}$ and $x_i$ on $V$ are related to this action.

In the case where $C=\{x^p=y^q\}$ for coprime $p$ and $q$ there is an action of the spherical rational Cherednik algebra of $SL_n$ with parameter $c=p/q$ on the cohomology of the compactified Jacobian of $C$ \cite{VV,OY1}, or rather its associated graded with respect to the perverse filtration, which is intimately related to the space $V$. For arbitrary torus links, it might still be true that $V$ or its variants carry some form of an action of a rational Cherednik algebra.

\end{remark}

\subsection*{Acknowledgements}

I would like to thank my advisor Eugene Gorsky for his keen interest in this work and his continuing support, as well as for suggesting to study homological correspondences on Hilbert schemes.
In addition, I thank Eric Babson, Brian Osserman, Vivek Shende and Zhiwei Yun for helpful discussions. 
This work was partially funded by the Vilho, Yrjö and Kalle Väisälä Foundation of the Finnish Academy of Science and Letters, a 
Fulbright Graduate Grant, and NSF grants DMS-1700814, DMS-15593381.

\section{Geometry of Hilbert schemes of points}\label{sec:geometry}
We describe the general setup for this paper.
Fix $C/\C$ a locally planar reduced curve and let $C=\bigcup_{i=1}^mC_i$ be a decomposition of $C$ to irreducible components. We will be working with 
versal deformations of $C$.

\begin{definition}
If $X$ is a projective scheme, a {\em versal deformation} of $X$ is a 
map of germs $\pi: \cX\to B$ such that $B$ is smooth, $\pi^{-1}(0)=X$, and given $\pi': 
\cX'\to B'$ with $\pi'^{-1}(b')=X$ there exists $\phi: B'\to B$ such that $b'
\mapsto 0$ and $\pi'$ is the pullback of $\pi$ along $\phi$. If $T_0B$ 
coincides with the first-order deformations of $X$, or in other words the 
base $B$ is of minimal dimension, we call $\pi$ a {\em miniversal} deformation.
\end{definition}

We call a family of locally planar reduced complex algebraic curves over a smooth base 
$B$ {\em locally versal} at $b\in B$ if the induced deformations of the germs of the singular 
points of $\pi^{-1}(b)$ are versal. We are interested in smoothness of relative families 
of Hilbert schemes of points for such deformations, needed 
for example for Lemma \ref{lemma-fundamentalclass}.

\begin{definition}
If $\cX\to B$ is any family of projective schemes, and $P(t)$ is any Hilbert polynomial, we denote the {\em relative Hilbert scheme} of this family by $\cX^{P(t)}$. By definition, Hilbert schemes are defined for families \cite{Groth}, and we note here that at closed points $b\in B$ the fibers of the relative Hilbert scheme are exactly $\Hilb^{P(t)}(\cX_b)$.
\end{definition}
We now consider the tangent spaces to (relative) Hilbert schemes.
\begin{lemma}\label{lem:normalbundle}
For any projective scheme $X$ and a flag of subschemes $X_1\subset\cdots\subset X_k$ in $X$ with fixed Hilbert polynomials $P_1(t),\ldots, P_k(t)$, the Zariski tangent space is given by
$$T_{(X_1,\ldots,X_n)}\Hilb^{\overrightarrow{P(t)}}(X)\cong H^0(X,\mathcal{N}_{(X_1,\ldots, X_m)/X}),$$
where the sections of the {\em normal sheaf}
$\mathcal{N}_{(X_1,\ldots, X_m)/X}\subseteq \bigoplus_{i=1}^k \mathcal{N}_{X_i/X}$ are tuples $(\xi_1,\ldots, \xi_k)$ of normal vector fields such that $\xi_i|_{X_j}=\xi_j$ modulo $\cN_{X_j/X_i}$ whenever $X_i\supseteq X_j$. The normal sheaf is by definition the sheaf of germs of commutative diagrams of homomorphisms of $\cO_X$-modules of the form 
\begin{equation}\begin{tikzcd}
\cI_k\arrow[d, "\sigma_k"] &\subset&\cI_{k-1}\arrow[d, "\sigma_{k-1}"]&\subset&\cdots&\subset&\cI_1\arrow[d, "\sigma_1"]\\
\cO_{X_k}&\rightarrow&\cO_{X_{k-1}}&\rightarrow&\cdots&\rightarrow&\cO_{X_1}
\end{tikzcd}
\end{equation}
\end{lemma}
\begin{proof}
Note that from first-order deformation theory it immediately follows that if $k=1$ we 
have $T_{X_1}\Hilb^{P(t)}(X)\cong H^0(\cN_{X_1/X},X)=H^0(X,\hom_{\cO_{X_1}}(\cI_1/
\cI_1^2, \cO_{X_1}))$, where $\cI_1$ is the ideal sheaf of $X_1$. For the proof 
of the result for flag Hilbert schemes we refer to \cite[Proposition 4.5.3]{Ser}.
\end{proof}

The following proposition is proved in e.g.
\cite[Proposition 17]{S}, and we reprove it here for convenience of the 
reader.
\begin{proposition}\label{prop:relhilbsm}
Let $\pi: \cC\to B'$ be a versal deformation of $C$, a reduced locally 
planar curve. Then 
the total space of the family $\pi^{[n]}:\cCn\to B'$ is smooth.
\end{proposition}
\begin{proof}

Let $B\subset \C[x,y]$ be a finite dimensional smooth family of polynomials 
containing the local equation for $C$ and all polynomials of degree at most $n$, such that the associated deformation is versal. Consider
the family of curves over $B$ given by $\cC_B:=\{(f\in B, p\in \C^2) 
f(p)=0\}$. Denote the fiber over $f$ by $C_f$ and let $Z\subset C_f$ be a subscheme of length $n$.
By e.g.
\cite[Section 4]{Ser}, there is always an exact sequence

$$0\to H^0(\cN_{Z/C_f},Z)\to T_ZC_B^{[n]}\to T_f B \to \Ext^1_{\cO_{C_f}}(\cI_Z,\cO_Z)$$

For squarefree $f$, there is always some open neighborhood $U$ of $f$ such that 
$C_U^{[n]}$ is reduced of pure dimension $n+\dim B$ \cite[Proposition 3.5]{MY}. Since $B$ is smooth and $H^0(\cN_{Z/C_f},Z)$ has dimension $n$ (see e.g. \cite{C}), it is enough to prove that the last $\Ext$-group vanishes to get smoothness of the total space $C_U^{[n]}$ at $Z$.

Now from the short exact sequence $$0\to \cI_Z\to \cO_{C_f}\to \cO_Z\to 0$$ taking $\hom$ to $\cO_Z$ we have 
$$\cdots \to \Ext^1_{\cO_{C_f}}(\cO_{C_f},\cO_Z)\to \Ext^1_{\cO_{C_f}}(\cI_Z,\cO_Z) \to \Ext^2_{\cO_{C_f}}(\cO_Z,\cO_Z)=0\to \cdots$$
As $\Ext^1_{\cO_{C_f}}(\cO_{C_f},\cO_Z)\cong H^1(\cO_Z,C_f)=0$ and the sequence is exact, we must have $\Ext^1_{\cO_{C_f}}(\cI_Z,\cO_Z)=0$ as well. So the total space is smooth.

Now if $\overline{\cC}\to \overline{B}$ is the miniversal deformation, 
by versality there are compatible isomorphisms $\cC\cong \overline{\cC}\times (\C^t,0)$ and $B\cong \overline{B}\times (\C^t,0)$ for some $t$, see e.g. \cite{GLS}. Hence we have smoothness for any versal family.

\end{proof}

We now consider the relative flag Hilbert scheme of a versal deformation. 
If $S$ is a smooth complex algebraic surface, its nested Hilbert scheme of points $S^{[n,n+1]}$ is smooth by 
results of 
\cite{C,T}. 
\begin{remark}
This nested Hilbert scheme $S^{[n,n+1]}$, together with the ordinary Hilbert scheme of $n$ points $S^{[n]}$ are the only flag Hilbert schemes of points on $S$ that are smooth, as shown in \cite{C,T}.
\end{remark}

We also have the following result.

\begin{proposition}\label{prop:relfhilbsm}
The total space of the relative family $\cCfn\to B$ is smooth.
\end{proposition}
\begin{proof}
This is a local question, so we can assume that $C$ is the germ of a plane curve singularity in $\C^2$.
    
From Lemma \ref{lem:normalbundle}, we have that
at $(J\subset I)\in \cCfn$ the tangent space is $\ker(\phi-\psi)$, where
\begin{align*}
\phi:\; &\hom_{\C[x,y]}(I,\C[x,y]/I)\to \hom_{\C[x,y]}(J,\C[x,y]/I), \;\;
\text{and}\\ 
\psi:\; &\hom_{\C[x,y]}(J,\C[x,y]/J)\to \hom_{\C[x,y]}(J, \C[x,y]/I)
\end{align*}
are the induced maps given by restriction and further quotient. 
Here $\phi-\psi$ is the difference of the maps from the direct sum. This is precisely the requirement needed for the normal vector fields in question.

Suppose again that $B\subset \C[x,y]$ is a finite dimensional smooth family of polynomials 
containing the local equation for $C$ and all polynomials of degree at most $n+1$, such that the associated deformation is versal.

Consider the inclusion $\cC^{[n,n+1]}_B\hookrightarrow B\times (\C^2)^{[n,n+1]}$. We have an exact sequence 
$$0\to T_{f,J\subset I}\cC_{B}^{[n,n+1]}\to 
T_fB\times T_{J\subset I}(\C^2)^{[n,n+1]}\to 
\C[x,y]/J$$ where the last map is given by $\left(f+\epsilon g,\begin{pmatrix}
\eta_1\\ \eta_2
\end{pmatrix}\right)\mapsto (\phi\eta_1)(f)-g$ mod $J$. By the assumption on $B$
the last map is surjective, hence $T_{f,J\subset I}\cC_B^{[n,n+1]}$ has dimension $\dim (\C^2)^{[n,n+1]}+\dim B-n+1=n+1+\dim B$, as expected, and the total space is smooth.

Again, if $\overline{\cC}\to \overline{B}$ is the miniversal deformation, 
by versality there are compatible isomorphisms $\cC\cong \overline{\cC}\times (\C^t,0)$ and $B\cong \overline{B}\times (\C^t,0)$ for some $t$, see e.g. \cite{GLS}. Hence we have smoothness for any versal family.

\end{proof}

Another result we will also need is the description of the components and dimensions of the irreducible components of $C^{[n]}$.

\begin{proposition}\label{prop-comp}
For any locally planar reduced curve $C=\bigcup_{i=1}^mC_i$, the irreducible 
components of $C^{[n]}$ are given by $$\overline{(C_1^{sm})^{[r_1]}\times
\cdots\times (C_m^{sm})^{[r_m]}}, \; \; \; \sum_i r_i=n.$$ Here $C^{sm}_i$ denotes the smooth locus of $C_i$.
In particular, there are $\binom{n+m-1}{n}$ irreducible components of $C^{[n]}$, all of dimension $n$.
\end{proposition}
\begin{proof}
That $(C^{sm})^{[n]}$ is dense in $C^{[n]}$ can be found e.g. in \cite[Fact 2.4]{MRV}. The schemes $(C_1^{sm})^{[r_1]}\times\cdots\times 
(C_m^{sm})^{[r_m]}$ are disjoint. They are of dimension $n$, smooth and 
connected, so irreducible, and as $(r_1,\ldots, r_m)$ runs over all 
possibilities, cover $(C^{sm})^{[n]}$. Taking closures we get the result.
\end{proof}

\section{Definition of the algebra A}\label{sec:definitions}
Let $V=\bigoplus_{i\geq 0} H_*(C^{[n]}, \Q)$, where we take singular 
homology in the analytic topology. This is mostly for simplicity, a majority of 
the results work with $\Z$ coefficients. A notable exception is Section \ref{sec:node}, where $\Q$-coefficients are essential. From now on we will be suppressing 
the coefficients from our notation. $V$ is naturally a bigraded $\Q$-vector 
space, graded by the number of points $n$ and homological degree $d$. We 
denote by $V_{n,d}$ the $(n,d)$-graded piece of $V$. We define the 
following operators on $V$, following ideas of Rennemo \cite{R} (and that 
originally go back to Nakajima and Grojnowski \cite{Nak,Groj}).

\begin{definition}\label{def:ops}
\hfill\\
\begin{enumerate}
\item Let $c_i\in C_i^{sm}$ be fixed smooth points and $\iota_i: C^{[n]}\to C^{[n
+1]}$ be the 
maps $Z\mapsto Z\cup c_i$. Let $x_i: V\to V$ be the operators given by $
(\iota_i)_*$. These are homogeneous of degree $(1,0)$ and only depend on the component the points $c_i$ lie in, see Lemma \ref{lemma-homotopic} below.

\item Let $d_i: V\to V$ be the operators given by the Gysin/intersection pullback map $(\iota_i)^!$. These are 
homogeneous of degree $(-1,-2)$, and well defined since the $\iota_i$ are 
regular embeddings. See Lemma \ref{regularembedding} below for a proof of 
this latter fact.
\end{enumerate}
\end{definition}

\begin{lemma}\label{lemma-homotopic}
The maps $\iota_{c_i}$ and $\iota_{c_i'}$ are homotopic whenever $c_i, c_i'\in C_i^{sm}$. In particular the corresponding pushforwards induce the same operators $x_i$ on $V$.
\end{lemma}
\begin{proof}
Take any path $c_i(t)$ from $c_i$ to $c_i'$
and consider the homotopy $C^{[n]}\times [0,1]\to C^{[n+1]}$ given  by 
$(Z,t)\mapsto Z\cup c_i(t)$.
\end{proof}

\begin{lemma}\label{regularembedding}
The map $\iota_x$ is a regular embedding.
\end{lemma}
\begin{proof}

This is a property which is local in the analytic topology (\cite[Chapter 2, Lemma 2.6]{ACG}). Suppose $Z\subset C$ is a subscheme of length $n$ 
which contains $x$ with multiplicity $k$.

If $U$ is an analytic open set around $x$ such that the only component of 
$Z$ contained in $\bar{U}$ (closure in the analytic topology) is $x$. Then 
locally around $Z$ the morphism is isomorphic to 
$$U^{[k]}\times (C\backslash \bar{U})^{[n-k]}\hookrightarrow (U)^{[k+1]}
\times (C\backslash \bar{U})^{[n-k]},$$
where the map is given on factors by adding $x$ and the identity map, 
respectively. The first map in local coordinates looks exactly like the 
inclusion $\C^{[k]}\to \C^{[k+1]}$ given as follows. If we identify 
coordinates on $\C^{[k]}$ with symmetric functions $a_i$ in the roots of 
some degree $k$ polynomial, i.e. coefficients of a monic polynomial of 
degree $k$, the map is given by $\sum_{i=0}^{k-1} a_iz^i\mapsto (z-x)
\sum_{i=0}^{k-1}a_iz^i$. But this last map is linear in the $a_i$ and of rank $k$, in particular a regular embedding.
\end{proof}

Consider the following diagram.
\begin{equation}\label{pushpulldiagram}
\begin{tikzcd}
&C^{[n,n+1]}\arrow[dl,"p"'] \arrow[dr,"q"]&\\
C^{[n]} & & C^{[n+1]}
\end{tikzcd}
\end{equation}

To define the operators $\mu_+$ and $\mu_-$, we want to define correspondences in homology between $C^{[n]}$ and $C^{[n+1]}$. This is done as follows.
By Propositions \ref{prop:relhilbsm}, \ref{prop:relfhilbsm}, we may embed $C$ into a smooth locally versal family $\pi: 
 \mathcal{C}\to B$ so that the relative family $\mathcal{C}^{[n]}$ is 
 smooth and $\pi^{-1}(0)=C$. After possibly doing an \'etale base extension, we may also assume that the family also has 
 sections $s_i: B\to \mathcal{C}$ hitting only the smooth loci of the 
 fibers and so that $s_i(0)=c_i$. 

Consider now the diagram
\begin{equation}\label{familydiagram}
\begin{tikzcd}
 & C^{[n,n+1]} \arrow[dl,"p"'] \arrow[dr,"q"] \arrow[dd,"i"] & \\
C^{[n]}\arrow[dd,"i"] & & C^{[n+1]}\arrow[dd,"i"]\\
 & \cC^{[n,n+1]} \arrow[dl, "\tilde{p}"'] \arrow[dr,"\tilde{q}"] & \\
 \cC^{[n]} & & \cC^{[n+1]}
\end{tikzcd}
\end{equation}
 where $i$ is the inclusion of the central fiber. 
 Since $\cC^{[n]}$ is smooth,
 from Property \ref{bivariantproperty:8} in Section \ref{sec:bivariant}, we have that 
 $$H^*(\cC^{[n,n+1]}\to \cC^{[n]})\cong H_{*-2n-\dim B}^{BM}(\cC^{[n,n+1]}).$$ 
Denote the fundamental class of $\cC^{[n,n+1]}$ under this isomorphism as $[\tilde{p}]$. Then pulling back $[\tilde{p}]$ along $i$ to
$H^*(C^{[n,n+1]}\to C^{[n]})$ gives us a canonical orientation, using which we define 
$p^!:H_*(C^{[n]})\to H_*(C^{[n,n+1]})$ as $p^!(\alpha)=\alpha \cdot i^*([\tilde{p}])$.
The definition of $q^!$ is identical, where we 
replace $C^{[n]}$ by $C^{[n+1]}$. 

We are finally ready to define the operators $\mu_+$ and $\mu_-$.

\begin{definition}\label{def:ops2}
Let $\mu_\pm: V \to V$ be the {\em Nakajima correspondences} $\mu_+=q_*p^!$, 
and $\mu_-=p_*q^!$. These are operators of respective bidegrees $(1,2)$ and $
(-1,0)$.
\end{definition}
\begin{remark}
 The $n$-degree in the above maps is easy to see from the definition. The homological degrees follow from the definition of the Gysin maps using 
 $i^*[\tilde{p}]$, which sits in homological degree $2n+2$, and the fact that degrees are additive under the bivariant product.
\end{remark}

We are now ready to define the algebra(s) $A$.
\begin{theorem}\label{thm-hardrelations}
The operators defined in Definitions \ref{def:ops}, \ref{def:ops2} satisfy the following commutation relations: $
[d_i,\mu_+]=[\mu_-,x_i]=1$, and the rest are trivial.
\end{theorem}
\begin{remark}
For $m=1$ we recover Theorem 1.2 in \cite{R}.
\end{remark}

\begin{definition}
Fix $m\geq 1$. Let $A_m$ be the $\Q$-algebra generated by the symbols $$x_1,\ldots, x_m,  d_1, \ldots, d_m, \mu_+, \mu_-$$ with the relations 
$$[d_i,\mu_+]=[\mu_-,x_i]=1, [x_i,x_j]=[x_i,d_i]=[x_i,\mu_+]=[d_i,\mu_-]=0.$$
\end{definition}
\begin{remark}
We can realize $A_m$ inside $\Weyl(\A^{2m}_\Q)$ as follows: Let $\A^{2m}$ 
have coordinates $x_1,\ldots,x_m,y_1,\ldots,y_m$ and $d_i=\partial_{y_i}, 
k=\sum_{i=1}^m \partial_{x_i}, j=\sum_{i=1}^m y_i$. Then from the commutation relations,
 we immediately have that $A_m$ is isomorphic to the 
subalgebra $\langle x_i, \partial_{y_i}, \sum_{i=1}^m \partial_{x_i},
\sum_{i=1}^m x_i\rangle\subset \Weyl(\A^{2m}_\Q)$.
\end{remark}
\begin{remark}
Although $A_m$ depends on $m$ we will be suppressing the subscript from the notation from here on. It should be evident from the context which $m$ we are considering.
\end{remark}

Let us give an outline of the proof of Theorem \ref{thm-hardrelations}. We first prove the 
trivial commutation relations in Subsections \ref{subsec-easy1}, 
\ref{subsec-easy2}. We then prove in Subsection \ref{sec:hardrels} that $[d_i,\mu_+]=1$ with the aid of the bivariant homology formalism, and then in a similar vein that $[\mu_-,x_i]=1$.

\subsection{Bivariant Borel-Moore homology}\label{sec:bivariant}

We now describe the bivariant Borel-Moore homology formalism from \cite{FM}. Suppose we are in a category of "nice" spaces; for example those that can be embedded in some $\R^n$.
We will not define bivariant homology here, but for us the most essential facts about it are the following ones:

\begin{enumerate}
\item The theory associates to maps $X\xrightarrow{f} Y$ a graded abelian group $H^*(X\xrightarrow{f} Y)$. We will be working over $\Q$ throughout also with bivariant homology. \label{bivariantproperty:1}

\item Given maps $X\stackrel{f}{\to}Y\stackrel{g}{\to}Z$, there is a product homomorphism
$$H^i(X\stackrel{f}{\to} Y) \otimes H^j(Y\stackrel{g}{\to}Z) \to H^{i+j}(X\stackrel{g \circ f}{\to}Z).$$
For $\alpha \in H^i(X\stackrel{f}{\to} Y)$ and $\beta \in H^j(Y\stackrel{g}{\to}Z)$ we thus get a product $\alpha \cdot \beta \in H^{i+j}(X\stackrel{g \circ f}{\to}Z)$.\label{bivariantproperty:2}

\item For any proper map $X \stackrel{f}{\to} Y$ and any map $Y \stackrel{g}{\to} Z$ 
there is a pushforward homomorphism $f_* : H^*(X\stackrel{g\circ f}{\to} Z) \to H^*(Y
\stackrel{g}{\to}Z)$.\label{bivariantproperty:3}

\item For any cartesian square
$$\begin{tikzcd}X' \arrow[r] \arrow[d,"g"] & X \arrow[d,"f"]\\
Y'\arrow[r] & Y\end{tikzcd}$$
there is a pullback homomorphism $H^*(X \stackrel{f}{\to} Y) \to H^*(X' \stackrel{g}{\to} Y')$. (Recall that a cartesian square is a square where $X'\cong X\times_Y Y'$.)\label{bivariantproperty:4}

\item Product and pullback commute: Given a tower of cartesian squares
$$\begin{tikzcd}X' \arrow[r,"h''"] \arrow[d,"f'"] & X \arrow[d,"f","\alpha"']\\
Y'\arrow[r,"h'"] \arrow[d,"g'"] & Y \arrow[d,"g","\beta"'] \\
Z' \arrow[r,"h"] & Z
\end{tikzcd}$$ we have $h^*(\alpha\cdot \beta)=h'^*(\alpha)\cdot h^*(\beta)$ in $H^*(X'\stackrel{g'\circ f'}{\to}Z')$.\label{bivariantproperty:5}

\item Product and pushforward commute: Given $$X\stackrel{f}{\to} Y\stackrel{g}{\to} Z\stackrel{h}{\to}W$$ with $\alpha\in H^*(X\stackrel{f\circ g} Z)$ and 
$\beta\in H^*(Z\stackrel{h}{\to}{W}$, $f_*(\alpha\cdot \beta)=f_*(\alpha)\cdot \beta$ in $H^*(Y\stackrel{h\circ g}{\to} W)$.\label{bivariantproperty:6}

\item For any space $X$, the groups $H^i(X \to \pt)$ and $H^i(X \stackrel{\id}{\to} X)$ 
are by construction canonically identified with $H^{BM}_{-i}(X)$ and $H^i(X)$, respectively. These are called the associated covariant and contravariant theories, respectively. Note that the three bivariant operations recover the usual homological operations of cup and cap product, proper pushforwards in homology and arbitrary pullbacks in cohomology.\label{bivariantproperty:7}

\item  If $Y$ is a nonsingular variety and $f: X \to Y$ is any morphism, the induced homomorphism $$H^*(X\stackrel{f}{\to} Y) \to H^{*-2\dim Y}(X\to \pt) = H^{BM}_{2\dim Y-*}(X)$$
given by taking the product with $[Y] \in H^{-2\dim Y}(Y\to \pt)$ is an isomorphism. Again the last equality is given by the associated covariant theory.
In such a situation we will frequently identify $H^{*}(X \to Y)$ with $H^{\mathrm{BM}}_{2\dim Y-*}(X)$.
In particular, if $X$ has a fundamental class $[X] \in H_{2\dim X}^{BM}(X)$, this induces a class $[X]\in H^{2(\dim Y - \dim X)}(X \to Y)$.
\item Any class $\alpha \in H^i(X \stackrel{f}{\to} Y)$ defines a Gysin pull-back map $f^!:H^{BM}_*(Y) \to H_{*-i}^{BM}(X)$ by $$f^!(\beta) = \alpha\cdot\beta, \ \ \ \ \ \forall \beta \in H^{\text{BM}}_*(Y).$$ \label{bivariantproperty:8}
\end{enumerate}

\section{Proof of the commutation relations}\label{sec:commrel}

\subsection{Proof of the trivial commutation relations for $x_i$ and $d_i$}\label{subsec-easy1}

We now show that $[x_i, x_j]=0$ for all $i,j$. This is fairly easy; under either composition $\iota_i\circ \iota_j, \iota_j\circ\iota_i$ we map
$Z\mapsto Z\cup x_i\cup x_j$ and as $(\iota_i\circ \iota_j)_*=x_ix_j$, we get $x_ix_j=x_jx_i$.

The next step is to describe the Gysin maps and their commutation relations. Denote these as before by 
$d_i=(\iota_{x_i})^!: H_*(C^{[n]})\to H_{*-2}(C^{[n-1]})$.

\begin{proposition}
We have $[d_i,d_j]=0$ and $[d_i,x_j]=0$.
\end{proposition}
\begin{proof}
Let $\alpha\in H_*(C^{[n]})$. As we saw before, $\iota_{x_i}\circ \iota_{x_j}=\iota_{x_j}\circ \iota_{x_i}$. By functoriality of the Gysin maps, $[d_i,d_j]=0$. 

Now choosing a representative for $\alpha$, we see that
$$d_ix_j(\alpha)=(\iota_{j})^![\iota_{i}(\alpha)]=[Z\in \alpha|c_i, c_j\subset Z].$$
However, we also have
$$x_jd_i(\alpha)=x_j[Z'\in\alpha|c_i\subset Z']=[Z\in \alpha|c_i,c_j\subset Z].$$
When the points are equal in the above, that is to say $i=j$, we pick a linearly equivalent point $c_i'$ near $c_i$. Since the inclusion maps $\iota_{c_i}, \iota_{c_i'}$ are homotopic in this case by Lemma \ref{lemma-homotopic}, we still have $[d_i,x_j]=0$.

\end{proof}

\subsection{Proof of the trivial commutation relations for Nakajima operators}\label{subsec-easy2}
As we saw before, the definition of the Nakajima operators requires making 
sense of the Gysin morphisms $p^!$ and $q^!$, which is done 
using the bivariant homology formalism. Recall that we are working with a 
fixed family $\cC\to B$ as in Section \ref{sec:definitions}, guaranteeing 
smoothness of $\cC^{[n]}$ and $\cC^{[n,n+1]}$. In this section and later 
on, all commutative diagrams should be thought of as commutative diagrams 
of topological spaces (corresponding to the analytic spaces of the 
varieties under consideration) and
living over $B$, such that we may restrict to the central fiber and obtain 
similar squares with calligraphic $\cC$s replaced by regular $C$s, i.e. our 
curve of interest. We will denote these restrictions in homology 
computations by the subscript $0$.

\begin{proposition}

We have $[x_i,\mu_+]=0$ and $[d_i,\mu_-]=0$.
\end{proposition}
\begin{proof}
Consider the following commutative diagram:

\begin{equation}\label{gysinpf}
\begin{tikzcd}
 & \mathcal{C}^{[n,n+1]} \arrow{rr}{\iota_{i}''} \arrow[swap]{dl}{p} \arrow{dr}{q} && \mathcal{C}^{[n+1,n+2]} \arrow{dr}{q'} \arrow[swap]{dl}{p'}\\ 
\cC^{[n]} \arrow{rr}{\iota_{i}}& & \mathcal{C}^{[n+1]} \arrow{rr}{\iota_{i}'}& & \mathcal{C}^{[n+2]}
\end{tikzcd}
\end{equation}

Here $\iota_i, \iota'_i$ are defined as adding points at the sections $s_i$ to $\cC^{[n]}, \cC^{[n+1]}$ respectively, and  $\iota''_i$ is adding points at the section $s_i$ as follows:
$(Z_1\subset Z_2)\mapsto (Z_1\cup s_i\subset Z_2\cup s_i)$.

We have by definition that $x_i\mu_+=((\iota_{s_i}')_*q_*p^!)_0$, where the 
subscript $0$ denotes restriction to the central fiber. 

In Diagram \ref{gysinpf}, the square formed by the maps
$\iota_i',q , q'$ and $\iota''_i$ is commutative, so on homology we have
$(\iota_{i}')_*q_*p^!=q'_*(\iota_{i}'')_*p^!$.
Similarly, the square formed by the maps $\iota_i,p',p,\iota_i''$ is commutative. Because of the fact that pushforward and the Gysin maps in bivariant homology also commute in this case, as explained in  Section \ref{sec:bivariant} (Property \ref{bivariantproperty:7}), we get
$q'_*(\iota_{i}'')_*p^!=q'_*(p')^!(\iota_{i})_*$. 
Restricting to $C$ we have $(q'_*(p')^!(\iota_{i})_*)_0=\mu_+x_i$.

Similarly, for the other commutation relation we have $\mu_-d_i=(p_*q^!(\iota_{i}')^!)_0=(p_*(\iota_{i}'')^!(q')^!)_0=((\iota_{i})^!p_*'(q')^!)_0=d_i\mu_-$.

Let us explain the restriction to central fiber once and for all. For example, in the last computation 
if $\alpha\in H_*(C^{[n]})$, we have 
$$\mu_- d_i \alpha=i^*[\tilde{p}]\cdot \iota_i^!\alpha=i^*[\tilde{p}]\cdot i^*[\tilde{d}_i]\cdot \alpha,$$
where $i^*[\tilde{d}_i]=[d_i]$ is the fundamental class corresponding to the Gysin map $\iota_i^!$ and $[\tilde{d}_i]$ is the corresponding class
in the family. Since the product and pullback commute in the bivariant theory,
$$i^*[\tilde{p}]\cdot i^*[\tilde{d}_i]\cdot \alpha=i^*([\tilde{p}]\cdot[\tilde{d}_i])\cdot \alpha.$$ Similarly, 
$d_i\mu_-\alpha=[d_i]\cdot i^*[\tilde{p}]\cdot \alpha =i^*([\tilde{d}_i]\cdot[\tilde{p}])\cdot \alpha$. So composing our operators in the family and then deducing the result for $C$ is justified.

\end{proof}

\subsection{Proof that $[\mu_-,x_i]=[d_i,\mu_+]=1$}\label{sec:hardrels}
To compute the desired commutation relation, we  compare the composition of 
the operators $d_i, \mu_+$ on $V$ in either order. By abuse of notation we will 
first consider $d_i$ and $\mu_+$ as operators acting on the space $\mathcal{V}=
\bigoplus_{n\geq 0} H_*^{BM}(\mathcal{C}^{[n]})$, and then use the 
properties of the bivariant theory, more precisely the ability to pull back 
in cartesian squares (Property \ref{bivariantproperty:5} in Section \ref{sec:bivariant}), to restrict to the special fiber and get an action 
on $V$.

Consider the diagrams

\begin{equation}\label{eq:square1}
\begin{tikzcd}
 & \mathcal{C}^{[n,n+1]} \arrow[d,"\kappa"', "q"] \arrow[dl,"\theta"',"p"] 
 & X_i \arrow[l,hook, "\widetilde{\lambda_i}"', "\widetilde{\iota_i}"] 
 \arrow[d,"\widetilde{\kappa_i}"', "\widetilde{q_i}"]\\ 
\cC^{[n]} & \mathcal{C}^{[n+1]} & \mathcal{C}^{[n]} 
\arrow[l,hook,"\lambda_i"',"\iota_i"]
\end{tikzcd}
\end{equation}
and
\begin{equation}
\begin{tikzcd}
 & &\mathcal{C}^{[n-1,n]} \arrow[dl,"\theta^\pr"',"p^\pr"] 
 \arrow[dr,"\kappa^\pr"',"q^\pr"]& \\
\mathcal{C}^{[n]} & \mathcal{C}^{[n-1]} \arrow[l,hook,"\lambda_i^
\pr"',"\iota_i^\pr"] & &\mathcal{C}^{[n]}
\end{tikzcd}
\end{equation}
where the two labels on the arrows denote the corresponding map $f: Y\to Z$ and a bivariant class
$\alpha\in H^*(X\stackrel{f}{\to} Y)$

In the first diagram 
$X_i = \mathcal{C}^{[n,n+1]}\times_{\mathcal{C}^{[n+1]}}\mathcal{C}^{[n]}$
 is the fiber product, and the square containing $X_i$ is cartesian. The 
morphisms $\iota_i, \iota_i^\pr$ correspond to adding a point at the 
sections $s_i: B \to \mathcal{C}$. The bivariant classes $\theta, 
\lambda_i, \kappa$ and their primed versions are the ones defined by 
fundamental classes, using the fact that the targets are smooth. The 
classes $\widetilde{\lambda_i}$ and $\widetilde{\kappa_i}$ are the cartesian 
pullbacks of $\lambda_i$ and $\kappa$, respectively.

Let $\alpha\in H_*(\mathcal{C}^{[n]})$. We first compute
\begin{align}\label{computation1}
d_i\mu_+(\alpha)&=\tilde{\lambda}_i\cdot q_*(\theta\cdot \alpha)=\tilde{q}_*(\lambda_i
\cdot \theta \cdot \alpha)\\
\mu_+d_i(\alpha)&=q'_*(\theta'\cdot \lambda_i'\cdot \alpha).\end{align}
Let us elaborate on the first computation a little bit. Here $\theta$ is 
the fundamental class and the isomorphism $H_{*-n-1}^{BM}(\cC^{[n,n+1]})\cong H^*(\cC^{[n,n+1]}\to \cC^{[n]})$ of the Borel-Moore homology group with the bivariant one is given by product with the 
fundamental class $\theta$. On the other hand, the Gysin pullback $\iota_i^!$ is by definition equal to the product in the bivariant theory 
with $\lambda_i$. In the first equation of \eqref{computation1}, we then use that the diagram in \eqref{eq:square1} in  is 
cartesian.

Let $f_i: \mathcal{C}^{[n-1,n]}\to X_i$ be given by $(Z_1\subset Z_2)\mapsto 
(Z_1 \cup s_i\subset Z_2\cup s_i, Z_1\cup s_i)$, and
$g_i: \mathcal{C}^{[n]}\to X_i$ be given by $Z\mapsto (Z\subset Z\cup s_i, Z)$.

\begin{lemma}\label{lemma-fundamentalclass}
For all $i$, $[X_i]=(f_i)_*[\mathcal{C}^{[n-1,n]}]+(g_i)_*([\mathcal{C}^{[n]}])$. 
\end{lemma}
\begin{proof}
By Proposition \ref{prop:relhilbsm} and Proposition \ref{prop:relfhilbsm} the total spaces of the relative families $\cC^{[n]}\to B$ and $\cC^{[n,n+1]}\to B$ are smooth.

Consider the fiber product $X_i$. The images of $f_i$ and $g_i$ cover all of 
$X_i$. On the level of points (of the fibers) this is easy to see; we are 
looking at pairs consisting of a flag of subschemes of lengths $n, n+1$  
and a subscheme of length $n$ that project to the same length $n+1$ 
subscheme in the cartesian square \eqref{eq:square1}. Since the points in the image contain 
$s_i$,  either the above pairs come from adding $s_i$ to both parts of the 
flag as well as taking the second factor to be $Z_1\cup s_i$, or by 
creating a new flag by adding $s_i$ to $Z$ and taking $Z$ to be the second 
factor.

 By \cite[Lemma 3.4]{R}, the intersection of the images of $f_i$ and $g_i$ 
 is codimension one in $X_i$. Consider a point$(Z\subset Z\cup s_i, Z)\in \im(f_i)\cap \im(g_i)$, which
 can also be written as $(Z'\cup s_i\subset Z'\cup s_i\cup s_i, Z'\cup s_i)$. We can then remove the smooth point $s_i$ from both of the factors unambiguously.
 So the intersection is isomorphic to $\cC^{[n-1]}$. On 
 the complement of the intersection the maps $f_i, g_i$ are scheme-theoretic 
 isomorphisms, because we can unambiguously remove the point $s_i$ from $(Z\subset Z\cup s_i, Z)$ or $(Z_1 \cup s_i\subset Z_2\cup s_i, Z_1\cup s_i)$. 
 Hence, the images of $f_i, g_i$ yield a partition of $X_i$ to 
 irreducible components. In particular, the fundamental class $[X_i]$ is the sum of 
 the fundamental classes of the images, which are by definition the 
 pushforwards in question.
\end{proof}

\begin{corollary}\label{cor:class}
We have $[X_i]=f_*(\theta'\cdot \lambda_i')+g_*[\cC^{[n]}]$.

\end{corollary}

\begin{proof}
 By Lemma \ref{lemma-fundamentalclass} $[X]=(f_i)_*[\cC^{[n-1,n]}]+
 (g_i)_*([\cC^{[n]}])$.
Rewrite $$[\cC^{[n-1,n]}]=\theta'\cdot \lambda_i'\in H^*(\cC^{[n-1,n]}\to \cC^{[n]})\cong H_{*-2n-2\dim B}^{BM}(\cC^{[n,n+1]}).$$ Then plugging this into the result of Lemma \ref{lemma-fundamentalclass} gives
$$[X_i]=f_*(\theta'\cdot \lambda_i')+g_*[\cC^{[n]}].$$
\end{proof}

Using Corollary \ref{cor:class},
we have 
\begin{equation}\label{nastycomputation}
d_ij(\alpha)=(\tilde{q}_{i})_*(\lambda_i\cdot \theta \cdot \alpha)=(\tilde{q}_{i})_*([X_i]\cdot\alpha)=
(\tilde{q}_{i})_*((f_i)_*(\theta'\cdot \lambda_i')\cdot \alpha)+(\tilde{q}_i)_*((g_i)_*[\cC^{[n]}]\cdot \alpha).\end{equation}
by the linearity of the pushforward. 
From Property \ref{bivariantproperty:6} in Section \ref{sec:bivariant}, we have that the pusforward is also functorial and commutes with products. Hence we have
\begin{equation}
(\tilde{q}_{i})_*((f_i)_*(\theta'\cdot \lambda_i')\cdot \alpha)+(\tilde{q}_i)_*((g_i)_*[\cC^{[n]}]\cdot \alpha)=(\tilde{q}_i\circ (f_i))_*(\theta\cdot \lambda_i'\cdot \alpha)+(\tilde{q}_i
\circ (g_i))_*(\alpha).\end{equation}
Finally,  the diagram in \eqref{eq:square1} is cartesian, and $\tilde{q_i}
\circ (g_i)=\id$, so
\begin{equation}
(\tilde{q}_i\circ (f_i))_*(\theta\cdot \lambda_i'\cdot \alpha)+(\tilde{q}_i
\circ (g_i))_*(\alpha)=q'_*(\theta'\cdot \lambda_i'\cdot \alpha)+\text{id}_*(\alpha)=jd_i(\alpha)+\alpha.
\end{equation}

Suppose now that $\alpha_0$ is a class in $H_*(C^{[n]})$. Then 
by the fact that pushforward, pullback, and the product in the bivariant theory commute,
$((\tilde{q}_{i})_0)_*((\lambda_i)_0\cdot \theta_0 \cdot \alpha_0)=(q'_0)_*((\theta')_0\cdot (\lambda_i')_0\cdot \alpha_0)+\text{id}_*(\alpha_0)$
and $d_i\mu_+=\mu_+d_i+\id: V\to V$, as desired.

The case of $[\mu_-,x_i]$ is much similar;
here we have
$$
\mu_-x_i(\alpha) = (p)_*(\kappa\cdot(\iota_i)_*(\alpha)) = (p\circ \widetilde{\iota}_i)_*(\widetilde{\kappa}_i\cdot\alpha)
$$
and
$$
x_i\mu_-(\alpha) = (\iota_i^\pr\circ p^\pr)_*(\kappa^\pr\cdot\alpha).
$$
Under the identification of $H^*(X \stackrel{\widetilde{q}}{\to} \mathcal{C}^{[n]})$ with $H_{*+2\dim \mathcal{C}^{[n]}}^{\mathrm{BM}}(X)$ 
we have $\widetilde{\kappa}_i = [X_i]$. This follows from
$$
\widetilde{\kappa}_i\cdot[\mathcal{C}^{[n]}] = \widetilde{\kappa}_i\cdot\lambda_i\cdot[\mathcal{C}^{[n+1]}] = \widetilde{\lambda}_i\cdot\kappa[\mathcal{C}^{[n+1]}] = \widetilde{\lambda}_i\cdot[\mathcal{C}^{[n,n+1]}] = [X_i],
$$
where the last equality is the fact that $X_i$ is a Cartier divisor in $\mathcal{C}^{[n,n+1]}$. Using Lemma \ref{lemma-fundamentalclass} we get
$$
\widetilde{\kappa}_i = [X_i] = (f_i)_*[\mathcal{C}^{[n-1,n]}] + (g_i)_*[\mathcal{C}^{[n]}] = (f_i)_*(\kappa^\pr) + (g_i)_*[\cC^{[n]}].
$$
A computation similar to \eqref{nastycomputation} now shows $\mu_-x_i(\alpha) = x_i\mu_-(\alpha) + \alpha$ as needed, and the restriction to the special fiber works exactly the same way. This finishes the proof of Theorem \ref{thm-hardrelations} and  thus of Theorem \ref{thm:mainthm}.

\section{Example: The node}\label{sec:node}
In this section we describe the representation $V$ for the the curve 
$\{xy=0\}\subseteq \P^2_\C$, which is the first nontrivial curve 
singularity 
with two components. % Its link is the Hopf link $T(2,2)$.
\subsection{Geometric description of $\Cn$}
One first thing we may ask is how the components in 
Proposition \ref{prop-comp} look like? Ran \cite{Ra} describes 
the geometry of the Hilbert scheme of points on (germs of) nodal curves 
very thoroughly. For $n=0,1$ we get a point and $C$ itself, whereas 
$C^{[2]}$ is a chain of three rational surfaces, that intersect their 
neighbors transversely along projective lines. More generally, $C^{[n]}$ is 
a chain of $n+1$ irreducible components of dimension $n$, consecutive 
members of which meet along codimension one subvarieties.
\begin{lemma}\label{lem:irredcomps}
Denote by $M_{n,k}$ the irreducible component of $C^{[n]}$, where generically we have  $k$ points on the component $y=0$ of $C$, and on the component 
$x=0$  we have $n-k$ points. Then $$
M_{n,k}\cong \text{Bl}_{\P^{k-1}\times \P^{n-k-1}}(\P^k
\times \P^{n-k}).$$
\end{lemma}
\begin{proof} 
%Consider the map $\pi: M_{n,k}=\overline{(C^{sm}_1)^{[k]}\times (C^{sm}%_2)^{[n-k]}}\to \P^k\times \P^{n-k}\cong C_1^{[k]}\times C_2^{[n-k]}$
%given on the smooth locus by restriction to components and extended by.

%Note that on the smooth locus this is also an isomorphism.
First of all, $\P^k\times \P^{n-k}$ has natural coordinates given by coefficients of polynomials $(a(x), b(y))$ of degrees $k$ and $n-k$.
It is also natural to identify the roots of these polynomials with the corresponding subschemes in $C_1, C_2\cong \P^1$.
From $(a(x),b(x))$ we construct an ideal in the homogeneous coordinate ring of $C$ by taking the product
$$I=(y,a(x))(x,b(x))=(xy, xa(x), yb(y), a(x)b(y)).$$ This determines a length $n$ subscheme so a point in $M_{n,k}$. Note that this map is invertible outside the locus where we have at least one point from each axis at the origin.

We can further write $a(x)b(y)=a_0b_0+a_0b'(y)+b_0a'(x)$ mod $(xy)$,
where $a(x)=a_0+a'(x)$, $b(y)=b_0+b'(y)$ and $a'(x), b'(x)$ have no constant term. 
Consider now the limit of $I=I_1$ as the products of the coordinates of the roots of $a(x), b(y)$ separately go to zero linearly, i.e. let
$t\to 0$ in $a_0=At$, $b_0=Bt$ and in the corresponding family of ideals $I_t$.
Since this is a flat family, the limiting ideal $I_0=\lim_{t\to 0}I_t$ has the same colength and support on the locus where at least one point from each axis is at the origin.
In particular $(a_0b_0+a_0b'(y)+b_0a'(x))/t \to Ab'(y)+Ba'(x)$ as $t\to 0$, and
$$\lim_{t\to 0}I_t=(xy, xa'(x), yb'(y), Ab'(y)+Ba'(x)).$$ Since all ideals 
in the locus of $M_{n,k}$ with at least one point from each axis at the origin can be written in this form, and $(A:B)\in \P^1$ determines the limiting ideal completely, we can identify $(A:B)$ with the normal coordinates $(a_0:b_0)$ and the natural map $$\pi: M_{n,k}=\overline{(C^{sm}_1)^{[k]}\times (C^{sm}_2)^{[n-k]}}\to \P^k\times \P^{n-k}$$ is 
the blowup along the locus where both $a(x)$ and $b(y)$ have zero as a 
root. 

See also \cite{Ran2} for a similar blow-up description.

\end{proof}

The intersections of the components can also be seen in this description.

\begin{lemma}
We have $E^n_{k,k+1}=M_{n,k}\cap M_{n,k+1}\cong \P^{n-k-1}\times \P^k$ and all the other intersections are trivial.
\end{lemma}
\begin{proof}

We continue in the notation of the proof of Lemma \ref{lem:irredcomps}.
Denote the locus where at least one point from either axis is at the origin in $M_{n,k}$ by $L_{n,k}$. Suppose then that we are outside $L_{n,k}\cup L_{n,\tilk}$ inside $M_{n,k}\cap M_{n,\tilk}$. Then only one point is at the origin and the this locus is identified naturally with the complement of the corresponding locus in $\P^{n-k-1}\times \P^k$ if $k+1=\tilk$ and is empty otherwise. We are thus left to studying the loci $L_{n,k}$. 

Consider again points $I=(xy, xa'(x),yb'(y)),Ab'(y)+Ba'(x))$ in $L_{n,k}$
and points $\tilde{I}=(xy, x\tila'(x),y\tilb'(y)),\tilA \tilb'(y)+\tilB \tila'(x))$ in $L_{n,\tilk}$.

First, restrict $I$ to the $x$-axis i.e. let $y=0$. Then $I|_{y=0}=(xa'(x),Ba'(y))$. If $B=0$ this has colength  $k+1$ since $a'(x)$ is of degree $k$. If $B\neq 0$ the colength is $k$.

Similarly, we get the colengths of $\tilde{I}|_{y=0}$ to be $\tilk$ or $\tilk+1$ depending on whether $\tilB$ is nonzero or not. Without loss of generality we can assume $\tilk>k$. In this case the only possibility for $I, \tilde{I}$ to be in the intersection $M_{n,k}\cap M_{n,k+1}$ is to have $k+1=\tilk, B=0$ and $\tilB\neq 0$.

A similar analysis for the $y$-axis shows that we must have $\tilA=0$ and $A\neq 0$. So in particular, the intersections $E^n_{k,\tilk}$ are isomorphic to $\P^{n-k-1}\times \P^k$ if $k+1=\tilk$ and empty otherwise.

\end{proof}
 
There is a heuristic way to think about the above lemma as well.
In $\Sym^{n}(C)$ the components $\P^{n-k}\times \P^k$ and $\P^{n-k-1}\times \P^{k+1}$ intersect when we move either a point from the $y$-axis to the origin in $\P^{n-k}\times \P^k$ or similarly for the $x$-axis in the next component. In particular, blowing up to get $M_{n,k}$ or $M_{n,k+1}$ the blow-up center is codimension one in the intersection, so 
nothing happens.

Now one may compute what $V$ is. There is a natural stratification of a blowup to the exceptional divisor and its complement. These both come with affine pavings, so a particularly easy way to compute the cohomologies of $C^{[n]}$, or at least the Betti numbers, is to count these cells. We have

\begin{proposition}
The bigraded Poincar\'e series for the space $V=\bigoplus_{n\geq 0} H_*(C^{[n]})$ is given by
$$P_V(q,t)=\frac{q^2t^2-q+1}{(1-q)^2(1-qt^2)^2}.$$

The grading corresponding to $t$ is the homological degree, whereas $q$ 
keeps track of the grading given by number of points.
\end{proposition}
\begin{proof}
It is easily confirmed that the Poincar\'e 
polynomials of the components are given by
$$P_{M_{n,k}}(t)=t^2\left(\sum_{i=0}^{k-1} t^{2i}\right)\left(\sum_{i=0}^{n-k-1} t^{2i}
\right)+\left(\sum_{i=0}^{n-k} t^{2i}\right)\left(\sum_{i=0}^{k}t^{2i}\right).$$ %See e.g. \cite{Bros} for generalities on motivic cell decompositions.
Similarly, the Poincar\'e polynomials of the intersections are given by 
$Q_{E^n_{k,k+1}}(t)=(\sum_{i=0}^{k}t^{2i})(\sum_{i=0}^{n-k-1}t^{2i})$, $k\leq n-1$, 
and $\sum_{k=0}^n P_{M_{n,k}}(t)-\sum_{k=0}^{n-1}Q_{E^n_{k,k+1}}(t)$ is by Mayer-Vietoris the Poincar\'e polynomial of $C^{[n]}$.  It is easy to see that $\sum_{n\geq 0}q^n\left(\sum_{k=0}^n P_{M_{n,k}}(t)-\sum_{k=0}^{n-1}Q_{E^n_{k,k+1}}(t)\right)=P_V(q,t)$.
\end{proof}

\begin{remark}
There is another way to see what the Poincar\'e polynomial is.

It is always true (see e.g. \cite[Section 3.2]{MSV}) that for a 
finite collection of points $P \subset C$, we 
have the following equality of series in the Grothendieck ring of varieties 
$\Var_\C[[t]]$:
\begin{equation}\label{eq:descent}
\sum_{n=0}^\infty q^n \cdot [C^{[n]}] = \left(\sum_{n=0}^\infty q^n \cdot [(C\setminus P)^{[n]}] \right) 
\cdot \prod_{p \in P} \left( \sum_{n=0}^\infty q^n  \cdot [C_p^{[n]}] \right)
\end{equation}

This for the case of a unique singular point reduces the computation to 
that of the punctual Hilbert scheme, and for multiple points contributions from their punctual Hilbert schemes. More precisely,
\end{remark}
\begin{proposition}
We have that $$\left(1+\sum_{m=1}^\infty q^m(1+(m-1)t^2)\right)\frac{1}{(1-qt^2)^2}=P_V(q,t).$$
\end{proposition}
\begin{proof}
The punctual Hilbert scheme $\Cn_0$ is by e.g. \cite{Ra} a chain of $n-1$ projective lines intersecting transversally, where the intersection points are given by the ideals $(x^i,y^{n-i})\subset \C[[x,y]]/xy$, $i=1,\ldots, n-1$, and the other ideals are given by $x^i+\lambda y^{n-i}$, $\lambda\in \C^*, i=1,\ldots, n-1$.  Taking for example 
the point corresponding to the ideal $(x,y^{n})$ as  the zero-cell
in $C^{[n]}_0$, and the $n-1$ affine lines given by the other ideals as the one-cells,  we get an affine paving of $C^{[n]}$. In other words, each affine cell is parameterized by $a, b, c\leq n$ with $a+b+c=n$ and $d\leq c$, where $a, b, c$ are the number of points on the smooth loci of the components as well as the number of points at the origin, and $d$ records the cell in the punctual Hilbert scheme. Since the homology of a variety with an affine paving is pure, by using \eqref{eq:descent} above, we get the formula $\left(1+\sum_{m=1}^\infty q^m(1+(m-1)t^2)\right)\frac{1}{(1-qt^2)^2}$. It is again easily checked that this equals $P_V(q,t)$ as desired.
\end{proof}

Figure \ref{table} shows the graded dimension of $V$ as a bigraded vector space. We can now compare this table to the cell decomposition given by the blow-up description above.

\begin{center}
\begin{figure}
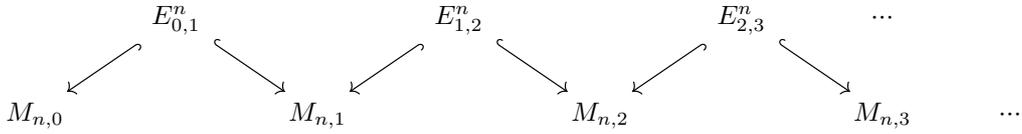

$$\begin{array}{|c|c|c|c|c|c|c|c|}
 \hline \;& \textbf{0} & \textbf{2} & \textbf{4} & \textbf{6} & \textbf{8} & \textbf{10} & \cdots \\ \hline
\textbf{0}& 1 & 0 & 0 & 0 & 0 &0 & \cdots \\ \hline
\textbf{1} & 1 & 2 & 0 & 0 & 0 & 0 & \cdots \\ \hline
\textbf{2} & 1 & 3 & 3 & 0 & 0 & 0 & \cdots \\ \hline
\textbf{3} & 1 & 4 & 5 & 4 & 0 & 0 & \cdots \\ \hline
\textbf{4} & 1 & 5 & 7 & 7 & 5 & 0 & \cdots \\ \hline
\textbf{5} & 1 & 6 & 9 & 10 & 9 & 6 & \cdots \\ \hline
\vdots&\vdots&\vdots&\vdots&\vdots&\vdots&\vdots&\ddots\\ \hline
\end{array}$$
\caption{The dimensions $V_{n,d}$ i.e. the Betti numbers of $C^{[n]}$. The 
columns are labeled by homological degree $d$ and the rows by the number of 
points $n$.}\label{table}
\end{figure}\end{center}

\subsection{Computation of the A-action}

We will now investigate the action of the algebra 
$A=\langle x_1,x_2,\mu_+,\mu_-,d_1,d_2\rangle_\Q$ on $V$.

Consider $V_{\bullet, 2n}=\bigoplus_i V_{i,2n}\subset V$, i.e. all the classes 
in homological degree $2n$. Denote by $[M_{n,k}]$ the fundamental class of 
the irreducible component $M_{n,k}$ of $C^{[n]}$ as described in the 
previous subsection.

\begin{theorem}\label{thm:fundclasses}
The fundamental classes $[M_{n,k}]\in V_{n,2n}$ generate $V_{\bullet,2n}$ as a $\Q[x_1,x_2]$-module.
\end{theorem}
\begin{proof}
This is equivalent to proving that the maps $x_i|_{V_{k,2n}}$ are jointly surjective for $k\geq n$. Dualizing the maps to pullbacks $x_i^*|_{V_{k,2n}}$ in cohomology, this condition is to say that the operators
$x_i^*: H^{<2k+2}(C^{[k+1]})\to H^*(C^{[k]})$ must satisfy $\bigcap \ker x_i^*=0$.

We have the following diagram for the components of $C^{[n]}$ and their intersections.

$$\begin{tikzcd}
 & E^{n}_{0,1} \arrow[dl, hook] \arrow[dr,hook] & & E^{n}_{1,2} 
 \arrow[dl, hook] \arrow[dr,hook] & & E^{n}_{2,3} \arrow[dl, hook] 
 \arrow[dr,hook] & \cdots &  \\
 M_{n,0} & & M_{n,1} & & M_{n,2} & & M_{n,3} & \cdots
\end{tikzcd}
$$

Since $C^{[n]}$ is a chain of the components $M_{k,n-k}$ 
intersecting transversally, without triple intersections, the Mayer-
Vietoris sequence in homology for unions splits to short exact sequences:

$$0\to \bigoplus_{k=0}^{n-1} H_i(E^n_{k,k+1})\to \bigoplus_{k=0}^n 
H_i(M_{n,k})\to H_i(C^{[n]})\to 0$$

Dually, we have an exact sequence the other way around in cohomology.
By our blow-up description of $\pi: M_{n,k}\mapsto 
\P^{n-k}\times \P^k$, we have as graded vector spaces 
that (see e.g. \cite{GH}), Chapter 6):

$$H^*(M_{n,k})=\frac{\pi^* H^*(\P^{n-k}\times \P^k)\oplus 
H^*(\P(\mathcal{N}_{\P^{n-k}\times \P^k/\P^{n-k-1}\times \P^{k-1}}))}{\pi^*H^*(\P^{n-k-1}\times \P^{k-1})}.$$
In particular, we can write $$H^*(\P^{n-k}\times \P^k)=\Q[a_{n,k}, 
b_{n,k}]/(a_{n,k}^{n-k+1}, b_{n,k}^{k+1})$$
as well as $$H^*(\P(\mathcal{N}_{\P^{n-k}\times \P^k/\P^{n-k-1}\times 
\P^{k-1}}))=
\Q[a_{n,k}',b_{n,k}',\zeta_{n,k}]/(a_{n,k}'^{n-k},b_{n,k}'^{k}, 
(\zeta_{n,k}-a_{n,k}')(\zeta_{n,k}-b_{n,k}')),$$
where $\zeta_{n,k}=c_1(\mathcal{O}(1))$.

Since the classes $a_{n,k}''^ib_{n,k}''j\in \pi^*H^*(\P^{n-k-1}\times \P^{k-1})\cong \Q[a_{n,k}'',b_{n,k}'']/(a_{n,k}''^{n-k},b_{n,k}''^k)$ are in $\pi^*H^*(\P^{n-k}\times \P^k)$ identified with $a_{n,k}^ib_{n,k}^j$ where $i<n-k, j<k$ and in the exceptional divisor with $a_{n,k}'^ib_{n,k}'^j$,
the quotient map as graded $\Q$-modules is the one identifying $a_{n,k}^ib_{n,k}^j$
with $a_{n,k}'^ib_{n,k}'^j$.

All in all, we can write as a graded $\Q$-vector space that 
\begin{equation}\label{eq:cohomology}
H^*(M_{n,k})=\langle 1,a_{n,k},b_{n,k},\ldots, a_{n,k}^{n-k}b^k_{n,k}, \zeta_{n,k}, \ldots, \zeta_{n,k} a_{n,k}^{n-k-1}b^{k-1}\rangle_\Q.
\end{equation}

%% insert affine paving stuff
%% look for typos "compare yadda yadda"
%%
%%

\begin{example}
We have that $H^*(M_{2,0})=\spann\{1,a_{3,0},a_{3,0}^2\}$, as expected, since $M_{2,0}\cong \P^2_\C$.
We also have that $H^*(M_{2,1})=\spann\{1,a_{2,1},a_{2,1}^2,b_{2,1},a_{2,1}b_{2,1}, a_{2,1}^2b_{2,1},\zeta_{2,1},\zeta_{2,1}a_{2,1}\}.$
\end{example}

Having described the cohomology of the components $M_{n,k}$ we can get back to our exact sequence. 
Identify $$H^*(\P^{n-k-1}\times \P^{k})\cong \Q[\mu_{n,k},\nu_{n,k}]/(\mu_{n,k}^{n-k}, \nu_{n,k}^{k+1}).$$

\begin{lemma}\label{gluing}
Under the inclusion $E^n_{k,k+1}\hookrightarrow M_{n,k}$
we have $\mu_{n,k}\mapsto a_{n,k}-\zeta_{n,k}$ and $\nu_{n,k}\mapsto b_{n,k}-\zeta_{n,k}$ in cohomology. Similarly, under the inclusion $E^n_{k,k+1}\hookrightarrow M_{n,k+1}$ we have
$\mu_{n,k}\mapsto a_{n,k+1}-\zeta_{n,k+1}$ and $\nu_{n,k}\mapsto b_{n,k+1}-\zeta_{n,k+1}$.
\end{lemma}
\begin{proof}
The class of $\mu_{n,k}$ in the intersection is the class dual to the the line $L_{n,k}^y$, where we fix all points in $E^n_{k,k+1}$ at the origin except for one at the $y$-axis. Similarly the class of $\nu_{n,k}$ is the line $L^x_{n,k}$ where we have but one point on the $x$-axis. Under the blowup $\pi_{n,k}: M_{n,k}\to \P^{n-k}\times \P^k$ the class of $L^y_{n,k}$ in $M_{n,k}$ is given by the total transform, which satisfies $[L^y_{n,k}]+\zeta_{n,k}=a_{n,k}$. The computation for the other three cases is nearly identical and we omit it.
\end{proof}

\begin{example}\label{ex:nistwo}
When $n=2$ the Hilbert scheme $C^{[2]}$ has the following components: 
$M_{2,0}\cong M_{2,2}\cong \P^2$ and $M_{2,1}\cong Bl_{pt}(\P^1\times \P^1)$. The intersections are $E^2_{0,1}\cong E^2_{1,2}\cong \P^1$.
The fundamental class of the first intersection is denoted $\mu_{2,0}$ 
and that of the second one is denoted $\mu_{2,1}$. 
Under the inclusion $E^2_{0,1}\hookrightarrow M_{2,0}$, the class $\mu_{2,0}$ is identified with $a_{2,0}$, and under the inclusion $E^2_{0,1}\hookrightarrow M_{2,1}$ with $a_{2,1}-\zeta_{2,1}$.
Similarly, under the inclusion $E^2_{1,2}\hookrightarrow M_{2,1}$ the class $\mu_{2,1}$ is identified with $a_{2,1}-\zeta_{2,1}$, whereas under the inclusion $E^2_{1,2}\hookrightarrow M_{2,2}$ it is identified with $a_{2,2}$.
\end{example}

As follows from the definition of the maps $\iota_i: C^{[n]}\to C^{[n+1]}$, we can consider them as restricted 
to $M_{n,k}$. They induce, by abuse of notation, maps in cohomology $x_i^*: 
H^*(M_{n+1,k+i-1})\to H^*(M_{n,k})$. We can describe these maps explicitly.
\begin{lemma}
In the basis of \eqref{eq:cohomology}, we have
$x_1^*: a^i_{n+1,k}b^j_{n+1,k}\mapsto a^i_{n,k}b_{n,k}^j$ and
$\zeta_{n+1,k}a^i_{n+1,k}b^j_{n+1,k}\mapsto \zeta_{n,k}a^i_{n,k}b_{n,k}^j$. 
Similarly, $x_2^*: a^i_{n+1,k+1}b^j_{n+1,k+1}\mapsto a^i_{n,k}b_{n,k}^j$ and
$\zeta_{n+1,k+1}a^i_{n+1,k+1}b^j_{n+1,k+1}\mapsto \zeta_{n,k}a^i_{n,k}b_{n,k}^j$.
\end{lemma}
\begin{proof}
We are adding one fixed smooth point i.e. embedding $C^{[n]}\hookrightarrow C^{[n+1]}$ as a divisor. Blowing down the components it is immediate that the $a$-classes go to the $a$-classes and the $b$-classes go to the $b$-classes.
We can treat the classes in the exceptional divisor separately,
where everything reduces again to embedding products of projective spaces as above.
In addition, we need that $x_1^*\zeta_{n+1,k+1}=\zeta_{n,k}$, which is saying that the normal bundle of the exceptional divisor of $M_{n+1,k+1}$ restricts to that of the exceptional divisor of $M_{n,k}$ under the embedding $\iota_1: M_{n,k}\to M_{n+1,k+1}$. In the notation of Lemma \ref{lem:irredcomps} we have that the map $\iota_1$ is on $M_{n,k}$ given by multiplying $a(x)$ by $x-c$ for some fixed $c\neq 0$. 
In particular, the centers of the blowups become identified, and the restriction of the normal bundle of the exceptional divisor of $M_{n+1,k+1}$ is the normal bundle of $M_{n,k}$.
\end{proof}

Having the above lemmas at our hands, we want to prove that the intersections of the kernels of the $x_i^*$ are only the fundamental classes.

The basic object of study here is the commutative diagram 
\begin{equation}
\begin{tikzcd}
\bigoplus_{k=0}^n H^{<2n}(M_{n,k}) \arrow[d] & H^{<2n}
(C^{[n]}) \arrow[l] \arrow[d]\\
\bigoplus_{i=0}^{n-1}H^*(M_{n-1,i}) & H^*(C^{[n-1]}) \arrow[l]
\end{tikzcd}
\end{equation}

We can explicitly describe the kernels on the left: for each $M_{n,k}$,
only the classes $\zeta_{n,k}a_{n,k}^{k-1}b_{n,k}^{n-k-1}$ are in their intersection.  In particular the intersection of the kernels is nonempty.
But this can be remedied on the right, as follows. By Lemma \ref{gluing} and the Mayer-Vietoris sequence, inside the intersection we can write for example that
\begin{align*}
  x_2^*(\zeta_{n,k}a_{n,k}^{k-1}b_{n,k}^{n-k-1})  = &
x_2^*(\zeta_{n,k})a_{n-1,k}^{k-1}b_{n-1,k}^{n-k-1} =  x_2^*(a_{n,k}-a_{n,k+1}+\zeta_{n,k+1})a_{n-1,k}^{k-1}b_{n-1,k}^{n-k-1}  \\
= & (a_{n-1,k-1}-a_{n-1,k})a_{n-1,k}^{k-1}b_{n-1,k}^{n-k-1}\neq 0.\end{align*}
In particular, we see that the image of the fundamental class of the exceptional divisor is also nonzero, i.e. it is not in the kernel and $\bigcap_i \ker x_i^*=0$ on the right.

This finishes the proof of the Theorem.

\end{proof}

Having Theorem \ref{thm:fundclasses} at our hands, we can finally restate Theorem \ref{thm:node}:

\begin{theorem}
Consider the following bigraded vector space: let $V''=\Q[x_1,x_2,y_1,y_2]$
with $x_i$ in degree $(1,0)$ and $y_i$ in degree $(1,2)$.
Consider the action of $A'=\Q\langle x_i, \partial_{y_i}, \sum y_i, \sum \partial_{x_i}\rangle$ on this space as differential operators, and let $U$ be the submodule $\Q[x_1,x_2,y_1+y_2](x_1-x_2)$. Define $V'=V''/U$.
Then $V\cong V'$ as $A\cong A'$-modules.
\end{theorem}
\begin{proof}
That $A'$ is isomorphic to $A$ in this case follows from the commutation relations. We can identify $V'$ and $V$ as $A$-modules as follows: let the monomial $y_1^iy_2^j/i!j!$ correspond to the fundamental class of $M_{i+j,i}$. It is then clear that on the diagonal $\bigoplus_{n\geq 0} V_{n,2n}$  the operators $d_1, d_2, \mu_+, \mu_-$ act as the corresponding differential operators in $A'$. Namely, the Gysin maps $d_1, d_2$ are given by intersection, from which it follows that $d_1 [M_{i+j,i}]=[M_{i+j-1,i-1}]$ and 
$d_2[M_{i+j,i}]=[M_{i+j-1,i}]$. This can be compared to the fact that for example $\partial_{y_1}y_1^iy_2^j/i!j!=y_1^{i-1}y_2^j/(i-1)!j!$.

By the commutation relations, $[d_1^{i+1},\mu_+]=(i+1)d_1^i$, so
$$[d_1^{i+1},\mu_+]y_1^iy_2^j/i!j!=(i+1)y_2^j/j!, \; 
[d_2^{j+1},\mu_+]y_1^iy_2^j/i!j!=(j+1)y_1^i/i!$$ and in particular $d_1^{i+1}\mu_+y_1^iy_2^j/i!j!=(i+1)y_2^j/j!$, $d_2^{j+1}\mu_+y_1^iy_2^j/i!j!=(j+1)y_1^i/i!$. Since $\mu_+y_1^iy_2^j/i!j!=
\sum_{k=0}^{i+j+1} c_ky_1^ky_2^{i+j+1-k}$ for some constants $c_k$, we must have $c_k=0$ unless $k=i$ or $k=i+1$, in which case we have $c_k=1/i!j!$.
This shows that $\mu_+$ can be identified with multiplication by $y_1+y_2$ on the diagonal, and below the diagonal since it commutes with the action of $x_1, x_2$. 
The operator $\mu_-$ acts on the diagonal as zero by degree reasons. An argument similar to the above shows that $\mu_-$ acts below the diagonal by $\partial_{x_1}+\partial_{x_2}$.

Since the maps $x_i$ are jointly surjective on the rows by Theorem \ref{thm:fundclasses}, we get a surjection 
$\phi: \Q[x_1,x_2,y_1,y_2]\twoheadrightarrow V.$ 
This is an $A$-module homomorphism by above. Its kernel contains $U$,
since $(x_1-x_2)\cdot 1=0$ and the actions of $x_i$ and $\mu_+$ commute with the $x_i$.

Consider then the graded dimensions/Poincar\'e series of $V'$ and $V$. 
We have $$P_{V'}(q,t)=P_{V''}(q,t)-P_U(q,t)=\frac{1}{(1-q)^2(1-qt^2)}-
\frac{q(1-qt^2)}{(1-q)^2(1-qt^2)^2}=P_V(q,t),$$ and since $\ker \phi \supseteq U$, we must have $\ker \phi=U$.
\end{proof}

\iffalse
\section{Further directions}
One may be tempted to define the $\partial_{x_i}$ and so forth separately. But the naive construction does not work. In fact, it {\em cannot} work,
because we would not be preserving the module structure for say, the node. 

In the integral case, there is an Abel-Jacobi map, using which it is possible to cut out the homology of the compactified Jacobian from $V$ as the intersections of the kernels of the $\partial_x, \partial_y$.
In this case, we still have Abel maps but they are not surjective.

In \cite{MSV} the 

\fi

\end{document}